\documentclass[11pt,a4paper]{article}
\usepackage{amsthm}
\usepackage{amscd}
\usepackage{mathrsfs}
\usepackage{amsfonts}
\usepackage{amssymb}
\usepackage{pifont}
\usepackage{amsmath}
\usepackage{enumerate}
\usepackage{booktabs}
\usepackage{array}
\usepackage[toc,page]{appendix}
\usepackage[all]{xy}
\usepackage{txfonts}
\usepackage{latexsym}
\usepackage{amstext}
\usepackage{fancyhdr}
\usepackage{graphicx,picinpar,epsf}
\usepackage{lastpage}
\usepackage{hyperref}
\usepackage{verbatim}
\usepackage{bbding}
\usepackage{longtable}
\usepackage{textcomp}

\usepackage{tikz}
\usetikzlibrary{%
  calc,
  fit,
  shapes,
  backgrounds,
  intersections
}

\newcommand{\db}{\bar{\Delta}}

\newcommand{\cccp}{(1,\frac{\text{ch}_1}{\text{ch}_0},\frac{\text{ch}_2}{\text{ch}_0})\text{-plane}}
\newcommand{\pp}{\textbf P^2}

\renewcommand\appendix{\par
  \setcounter{section}{0}
  \setcounter{subsection}{0}
  \setcounter{figure}{0}
  \setcounter{table}{0}
  \renewcommand\thesection{Appendix \Alph{section}}
  \renewcommand\thefigure{\Alph{section}\arabic{figure}}
  \renewcommand\thetable{\Alph{section}\arabic{table}}
}

\newtheorem{theorem}{Theorem}[section]
\newtheorem{defn}[theorem]{Definition}
\newtheorem{prop}[theorem]{Proposition}
\newtheorem{propd}[theorem]{Proposition and Definition}
\newtheorem{cor}[theorem]{Corollary}
\newtheorem{lemma}[theorem]{Lemma}
\newtheorem{rem}[theorem]{Remark}
\newtheorem{nota}[theorem]{Notation}
\newtheorem{eg}[theorem]{Example}
\newtheorem{conj}[theorem]{Conjecture}

\def\C{\ensuremath{\mathbb{C}}}

\def\R{\ensuremath{\mathbb{R}}}
\def\Z{\ensuremath{\mathbb{Z}}}

\def\ch{\mathop{\mathrm{ch}}\nolimits}

\def\Coh{\mathop{\mathrm{Coh}}\nolimits}

\def\Hom{\mathop{\mathrm{Hom}}\nolimits}

\def\Ker{\mathop{\mathrm{Ker}}\nolimits}

\def\min{\mathop{\mathrm{min}}\nolimits}

\def\glt{\tilde{\mathrm{GL}}^+(2,\mathbb R)}
\def\HH{\mathrm{H}}
\def\KK{\mathrm{K}}

\newcommand{\pkp}{\mathrm{P}\left(\mathrm{K}_{\mathbb R}(\mathbf {P}^2)\right)}

\def\Stab{\mathop{\mathrm{Stab}}\nolimits}

\title{The Space of Stability Conditions on the Projective Plane}
\author{Chunyi Li}
\date{\today}

\def\ch{\mathrm{ch}}
\def\pp{\mathbf{P}^2}
\def\ppp{\mathbf{P}^1}

\def\cccp{\{1,\frac{\ch_1}{\ch_0},\frac{\ch_2}{\ch_0}\}\mathrm{-plane}}
\def\Db{\mathrm{D}^{\mathrm{b}}}

\def\C{\ensuremath{\mathbb{C}}}

\def\R{\ensuremath{\mathbb{R}}}
\def\Z{\ensuremath{\mathbb{Z}}}

\def\cE{\ensuremath{\mathcal E}}
\def\cF{\ensuremath{\mathcal F}}

\def\cH{\ensuremath{\mathcal H}}

\def\cO{\ensuremath{\mathcal O}}
\def\cP{\ensuremath{\mathcal P}}

\def\cT{\ensuremath{\mathcal T}}

\def\Std{\Stab^{\dag}}
\def\sdp{\Std(\pp)}

\def\Geo{\mathrm{Geo}}
\def\Alg{\mathrm{Alg}}
\def\MZ{\mathrm{MZ}}
\def\TR{\mathrm{TR}}

\begin{document}
\maketitle
\begin{abstract}
The space of Bridgeland stability conditions on the bounded derived
category of coherent sheaves on $\pp$ has a principle connected
component $\Stab^\dag$(\textbf P$^2$). We show that Stab$^\dag$(\textbf P$^2$) is the union of geometric and algebraic
stability conditions. As a consequence, we give a cell
decomposition for Stab$^\dag$ (\textbf P$^2$) and show that
Stab$^\dag$(\textbf P$^2$) is contractible.
\end{abstract}

\section*{Introduction}
Motivated by the concept of $\Pi$-stability condition on string theory by Douglas, the notion of a stability condition, $\sigma=(\mathcal P,Z)$, on a $\mathbb C$-linear triangulated category $\mathcal T$ was first introduced by Bridgeland in \cite{Bri07}. In the notion, the central charge $Z$ is a group homomorphism from the numerical Grothendieck group K$_0(\cT)$ to $\C$.  Bridgeland proves that the space of stability conditions inherits a natural complex manifold structure via local charts of central charges in $\Hom_{\Z}(\KK_0(\cT),\C)$. In particular, when K$_0(\cT)$ has finite rank, the space of stability condition (satisfying support condition),  $\Stab(\cT)$, has complex dimension rank(K$_0(\cT)$).


As mentioned in \cite{BriSS}, $\Stab(\cT)$ is expected to be related to the study of string theory and mirror symmetry. The main interesting example is to understand the space of stability conditions on a compact Calabi-Yau threefold $X$ such as a quintic in $\mathbf P^4$. Yet this problem is still wildly open mainly due to some technical difficulties.
Although the compact Calabi-Yau threefold case is still difficult to study, $\Stab(\cT)$ of various analog categories has been very well understood, see \cite{BSW, BQS, DK16, Ike14, Qiu15}. While most of these examples are build from quivers or locally derived category of sub-varieties, few cases of $\Stab(X)$ for smooth compact varieties $X$ are known. Such $\Stab(X)$ is `well-understood' only when $X$ is $\ppp$(\cite{Ok}), a curve (\cite{Bri07}), a K3 surface (\cite{Bri08,BB}), an abelian surface or threefold (\cite{BMS}). In this paper, based on some important technical results from \cite{Mac05} and \cite{Mac}, we make an attempt to analyze the space $\Stab(\pp)$.

\begin{theorem}[Theorem \ref{thm:stabdag}, Corollary \ref{cor:main}]
Let $\Std(\pp)$ be the connected component in $\Stab(\pp)$ that contains the geometric stability conditions, then
\[\Std(\pp) = \Stab^{\Geo}(\pp) \bigcup \Stab^{\Alg}(\pp).\]
In particular, $\Std(\pp)$ is contractible.
 \label{thm:main}
\end{theorem}

Here $\Stab^{\Geo}(X)$ denotes the space of geometric stability conditions (Definition \ref{defn:stabgeo}), at where the sky-scraper sheaves are stable with the same phase. $\Stab^{\Alg}(\pp)$ denotes the space of algebraic stability conditions (Definition \ref{def: theta E}), which can be constructed from exceptional collections.

\textbf{Rough description for $\sdp$:} We first describe the geometric part $\Stab^{\Geo}(X)$. When $X$ is a smooth surface, by the philosophy of \cite{Bri08} and \cite{BB}, $\Stab^{\Geo}(X)$ can be determined once people know the Chern characters of Gieseker-stable sheaves. The $\glt$-action (see Lemma 8.2 in \cite{Bri07}) acts freely on the part of $\Stab^{\Geo}(X)$. Any point in $\Stab^{\Geo}(X)/\glt$ is uniquely determined by the kernel of its central charge, which is a linear subspace in K$_\R(X)$ of real codimension two. From the inverse side, such a linear subspace can be realized as the kernel of a central charge if and only if one can construct a quadratic form $Q$ on K$_\R(X)$ satisfying the support condition (see the definition above Definition \ref{defn:stabgeo}) for this subspace. In the  case that $X$ is of Picard number one, one may take the projectivization of K$_\R(X)$, $\Ker Z$ is a point on P($\KK_\R(X)$). A point on $\pkp$ can be the kernel of a central charge if and only if it has an open neighborhood which is not `below' any Gieseker stable character. 

Now we focus on the case that $\cT$ is D$^b(\pp)$. The (projective) Gieseker stable characters have been completely determined in \cite{DP} by Drezet and Le Potier. On $\pkp$, the characters form a dense set below the Le Potier curve (see Definition \ref{lpcurve}) together with some isolated points of exceptional characters.

For the algebraic part $\Stab^{\Alg}(\pp)$, it goes back to the work \cite{Be} that D$^b(\pp)$ can be generated by an exceptional collection $\{ \cO,\cO(1), \cO(2)\}$. One can do mutations between the exceptional objects to get other exceptional triples, such as $\{ \cO(1),\cT_{\pp}, \cO(2)\}$, $\{ \cO(-4),\cO(-3), \cO(-2)\}$, which also generates the category. For each exceptional triple $\cE=\langle E_1,E_2,E_3\rangle$, one may assign numbers $z_j=m_j\exp(i\pi \phi_j)$, $\phi_j$ as the central charges and phases of $E_j$. Due to the result in $\cite{Mac}$,  when $m_j\in \R_{>0}$, $\phi_1<\phi_2<\phi_3,\text{ and }\phi_1+1<\phi_3$, there is a unique stability condition with the given central charge and $E_i\in \cP(\phi_i)$. Denote all such stability conditions by $\Theta_\cE$ with parameters $m_j$ and $\phi_j$.   The space of algebraic stability conditions $\Stab^{\Alg}(\pp)$ is the union of all $\Theta_\cE$. Note that the $\glt$-action does not act freely on $\Stab^{\Alg}(\pp)$. Each $\Theta_\cE$ can be divided into three parts: the head $\Theta_\cE^{\Geo}$; the legs $\Theta^{+}_{\mathcal E,E_1}$,
 $\Theta^{-}_{\mathcal E,E_3}$; and the tail $\Theta^{\mathrm{Pure}}_{\mathcal E}$ (see Definition \ref{def: theta E}). The head part is the overlap part with the geometric stability conditions, this is the only part of $\Theta_\cE$ that `glues' on the $\Stab^{\Geo}(X)$. The leg part overlaps with other algebraic stability conditions, we will show that any two legs of $\Theta_\cE$ and $\Theta_{\cE'}$ are either the same, or separated from each other (see Proposition \ref{prop:legsareunique}). Each tail part $\Theta^{\mathrm{Pure}}_{\mathcal E}$ is a private area for $\Theta_\cE$, which is separated from any other $\Theta_{\cE'}$ (see Lemma \ref{lemma:purepure}). We will show that one may contract the whole space of  $\Stab^{\Alg}(\pp)$ by first contracting all the tails simultaneously to their boundaries with legs, and then contracting all the legs to their boundaries with heads. The union of all heads $\bigcup \Theta_\cE^{\Geo}$ is a $\glt$-bundle over an open subset of $\Stab^{\Geo}(\pp)$, which is contractible.

\textbf{Related works:} Many important technical results  on $\Stab^{\Alg}(\pp)$ have been set up in \cite{Mac05} and \cite{Mac}, and our result is a natural continuation of the previous work. The space $\sdp$ can be compared with some previous geometric examples such as $\Std$(K3) and $\Std$(local $\pp$). As described in the previous section, their geometric parts $\Stab^{\Geo}(X)$ are quite similar.  In addition, each exceptional/spherical object provides two boundary sets of $\Stab^{\Geo}(X)$. But the remaining parts are very different,  for a K3 surface or local $\pp$, the remaining parts can be viewed as copies of the geometric part. While for $\sdp$, the remaining parts are similar to the space of stability conditions of quivers representations, see the works of \cite{BSW, BQS, DK16, Ike14, Qiu15, QW}. In most of the previous quiver representation examples, the stability conditions are all of the algebraic type. Yet the quiver representation for $\Db(\pp)$ has a complicated relation, this leads the fact that some of the geometric stability conditions on $\pp$ are not of the algebraic type.  In addition, it seems to the author that the contractibility of the algebraic parts  $\Stab^{\Alg}(\pp)$ is not implied by the results in any of the previous papers. In particular, the paper \cite{QW}, at where the authors prove the contractibility for many interesting examples, does not apply to the case $\Stab^{\Alg}(\pp)$,  since the heart $\langle \cO[2],\cO(1)[1],\cO(2)\rangle$ is not locally finite and has infinitely many algebraic tilts, which are crucial assumptions on the t-structure in \cite{QW}.

\textbf{Open questions:}
It is reasonable for us to believe that $\Stab^\dag(\pp)$ actually contains all the stability conditions that satisfy the support condition. 
\begin{conj} 
We expect the following statement holds:
$\Stab(\pp)=\Std(\pp)$.
\label{conj}
\end{conj}
In addition, as the case of $\ppp$, we wish to understand the global complex structure
of Stab$(\textbf P^2)$. We expect that there is a period map as that of the CY quiver
cases, \cite{BQS, Ike14}, so that we may
have differential forms on Stab$^\dag$(\textbf P$^2$) and the
central charge is neatly computed as integrations. But this seems difficult
to realize because there is some `pure geometric' part on
Stab$^\dag$(\textbf P$^2$). For the algebraic part Stab$^{\Alg}(\pp)$, we also expect that there is
a fundamental domain $R$ on $(\cH)^3\simeq \Theta_\cE$
independent of the triples $\mathcal E$ such that all the $R_\cE$'s form a
disjoint cover of Stab$^{\Alg}(\pp)$.\\

\textbf{Acknowledgments.} The author is grateful to Arend Bayer, Zheng Hua, Yu Qiu and Xiaolei Zhao for helpful conversations. The author is supported by ERC starting grant no. 337039 ``WallXBirGeom''.

\section*{Notations}
The Picard group of \textbf P$^2$ is of rank one with generator $H$ $=$ $[\mathcal O(1)]$, and we will, by abuse of notation, identify the $i$-th Chern character $\ch_i$ with its degree $H^{2-i}\ch_i$. The slope $\mu$ of a non-torsion sheaf $E$ on $\pp$ is defined as $\frac{\ch_1}{\ch_0}$. We denote $\mathrm{K}(\pp)\otimes \mathbb R$ by K$_{\mathbb R}(\pp)$. Consider the real projective space $\pkp$ with homogeneous coordinate $[\ch_0,\ch_1,\ch_2]$, we view the locus $\ch_0=0$ as the line at infinity. The complement forms an affine real plane, which is referred to as the $\cccp$. We call $\pkp$ the projective $\cccp$. For any object $F$ in $\mathrm D^b(\pp)$, we write 
\[\tilde{v}(F):=\big(\ch_0(F),\ch_1(F),\ch_2(F)\big)\]
as the numerical character of $F$, and $v(F)$ the projection of $\tilde{v}(F)$ on the $\cccp$ with locus $(1,s,q)$. 

Let $E$, $F$ be two objects in D$^b(\textbf P^2)$ with characters on the $\cccp$ and $P$
be a point on the projective $\cccp$. For the convenience of the reader, we make the
list of notations and symbols that are commonly used in this
article. Most of them are explicitly defined at other places of the
article.
\begin{longtable}{ p{.20\textwidth}  p{.65\textwidth} } 
\hline
   $\mathcal H_P$ & the right half plane  with $\frac{\ch_1}{\ch_0}>s$, or
$\frac{\ch_1}{\ch_0}=s$ and
      $\frac{\ch_2}{\ch_0}>q$
\\
     $\mathcal H_E$ & $H_{v(E)}$ when $v(E)$ is not at infinity\\

     $L_{EF}$ & the line on
     $\pkp$ across $v(E)$
     and $v(F)$\\
     $L_{EP}$ & the line on $\pkp$ across $v(E)$
     and $P$\\
     $l_{EF}$ ($l_{EP}$) & the line segment $\overline{v(E)v(F)}$ ($\overline{v(E)P}$)on
      the $\cccp$ \\
    $l_{EF}^r$ & the ray along $L_{EF}$ from $v(F)$ to infinity and does not contain $v(E)$\\
    $l_{EF}^+$ & the ray along $L_{EF}$ from $v(E)$ on the $\mathcal H_E$ part\\
    $l_{E+}$ & the ray segment on $L_{E(0,0,1)}$ on the $\mathcal H_E$ part\\
    $l_{E-}$ & the ray segment on $L_{E(0,0,-1)}$ outside the  $\mathcal H_{v(E)}$ part\\
$\mathcal E$ & a triple of ordered exceptional objects
$\{E_1,E_2,E_3\}$\\
TR$_{\mathcal E}$ & the inner points in the triangle bounded by
$l_{E_iE_j}$, for $1\leq i < j\leq 3$.\\
$e^*_i$ & $v^*(E_i)$ as defined in section 1, $*$ can be $+$, $l$,
$r$ or blank\\
MZ$_\mathcal E$ & the inner points of region bounded by
$l_{e_1e^+_1}$, $l_{e^+_1e_2}$,$l_{e_2e^+_3}$, $l_{e^+_3e_3}$ and
$l_{e_3e_1}$ \\ \hline
\caption{List of Notations} 
\label{table:listnot}
\end{longtable}

\section {Geometric stability conditions}
\subsection{Review: Exceptional objects, triples, and Le Potier curve}
Let $\mathcal T$ be a $\mathbb C$-linear triangulated category of
finite type. For convenience, one may always assume that  $\mathcal T$ is
$\Db(\pp)$: the bounded derived category of coherent sheaves
on the projective plane over $\mathbb C$. The following definitions
follow from \cite{AKO,GorRu, Or}.
\begin{defn}
An object $E$ in $\mathcal T$ is called \emph{exceptional} if
\begin{center} $\Hom^i(E,E) = 0,$ \text{ for } $i\neq 0$;
$\Hom^0(E,E)=\mathbb C$.\end{center} An ordered collection of
exceptional objects $\mathcal E = \{E_0,\dots,E_m\}$ is called an
\emph{exceptional collection} if
\begin{center}
$\Hom^{\bullet} (E_i,E_j)=0$, for $i>j$.
\end{center}
\end{defn}
\begin{defn}
Let $\mathcal E$ $=$ $\{ E_0,\dots,E_n\}$ be an exceptional
collection. We call this collection $\mathcal E$ \emph{strong}, if
\[\Hom^q(E_i,E_j) = 0,\] for  all $i$, $j$ and $q\neq 0$.
This collection $\mathcal E$ is called \emph{full}, if $\mathcal E$
generates $\mathcal T$ under homological shifts, cones and direct
summands.
\end{defn}

\def\epm{E_{\left(\frac{p}{2^m}\right)}}

We summarize some of the classification results of the exceptional
bundles on \textbf P$^2$ and make some notations, see \cite{DP, GorRu, LeP}.
 There is a one-to-one correspondence
between the dyadic integers $\frac{p}{2^m}$ and exceptional bundles $\epm$.
Let the Chern  character of the exceptional bundle corresponding to
$\frac{p}{2^m}$ be \[\tilde{v}(\epm):=
\left(\ch_0(\epm),\ch_1(\epm),\ch_2(\epm)\right),\]
the characters are inductively given by the formulas:
\begin{itemize}
\item $\tilde{v}(E_{(n)})$ $=$ $\left(1,n,\frac{n^2}{2}\right)$, for $n\in \mathbb Z$.
\item When $q>0$ and $p\equiv 3 $(mod $4$), the character is given
by
\[\tilde{v}\left(\epm \right) = 3\ch_0\left(E_{\left(\frac{p+1}{2^m}\right)}\right)\tilde{v}\left(E_{\left(\frac{p-1}{2^m}\right)}\right)-\tilde{v}\left(E_{\left(\frac{p-3}{2^m}\right)}\right).\]
\item When $q>0$ and $p\equiv 1 $(mod $4$), the character is given
by
\[\tilde{v}\left(\epm \right) = 3\ch_0\left(E_{\left(\frac{p-1}{2^m}\right)}\right)\tilde{v}\left(E_{\left(\frac{p+1}{2^m}\right)}\right)-\tilde{v}\left(E_{\left(\frac{p+3}{2^m}\right)}\right).\]
\end{itemize}
\begin{eg} Here are some first observations from the definition.
\begin{enumerate}
\item When $k\in \Z$, $\tilde{v}(E_{(k)})$ is the character for the line bundle $E_{(k)}=\cO_{\pp}(k)$.
 \item $\tilde{v}\left(E_{\left(\frac{3}{2}\right)}\right)$ is the character for the tangent
bundle $E_{\left(\frac{3}{2}\right)}=\cT_{\pp}$. \item The exceptional bundle
$E_{\left(\frac{p}{2^m}+1\right)}$ associates $\frac{p}{2^m}+1$ is
$E_{\left(\frac{p}{2^m}\right)}\otimes \mathcal O_{\pp}(1)$.
\end{enumerate}
\end{eg}

\textbf{Le Potier curve:} 
Define $v\left(\epm\right)$ $=$
$\tilde{v}\left(\epm\right)/\ch_0\left(\epm\right)$. 
We use Chern characters $[\ch_0,\ch_1,\ch_2]$ for the coordinate of
K$_{\mathbb R}(\textbf P^2)$. Consider the real projective space $\pkp$ with homogeneous coordinate $[\ch_0,\ch_1,\ch_2]$. We view the locus $\ch_0=0$ as the line at infinity, and call $\pkp$ the projective $\cccp$. The complement of the line at infinity forms an affine real plane, which is referred to as the $\cccp$. We  will define the Le Potier curve on this $\cccp$.

 Let
$e\left(\frac{p}{2^m}\right)$ be the point on the $\cccp$ with coordinate
$v\left(\epm\right)$.  We associate three more points
$e^+\left(\frac{p}{2^m}\right)$, $e^l\left(\frac{p}{2^m}\right)$ and $e^r\left(\frac{p}{2^m}\right)$
to $\epm$ on the
$\cccp$. The coordinate of
$e^+\left(\frac{p}{2^m}\right)$ is given as:
\[e^+\left(\frac{p}{2^m}\right)\; := \; e\left(\frac{p}{2^m}\right)\; -
\; \left(0,0,\frac{1}{\left(\ch_0\left(\epm\right)\right)^2}\right). \] For any real number $a$,
let $\Delta_{a}$ be the parabola:
\[\frac{1}{2}\left(\frac{\ch_1}{\ch_0}\right)^2-\frac{\ch_2}{\ch_0}\; = \; a\]
on the $\cccp$. Let $\Delta_{>a}$($\Delta_{<a}$) be the region 
$\left\{\left(1,\frac{\ch_1}{\ch_0},\frac{\ch_2}{\ch_0}\right)\middle|\frac{1}{2}
\left(\frac{\ch_1}{\ch_0}\right)^2- \frac{\ch_2}{\ch_0}>a(<a)\right\}$. The point
$e^l\left(\frac{p}{2^m}\right)$ is defined to be the intersection of
$\Delta_{\frac{1}{2}}$ and the segment
$l_{e^+\left(\frac{p}{2^m}\right)e\left(\frac{p-1}{2^m}\right)}$; $e^r\left(\frac{p}{2^m}\right)$ is
defined to be the intersection of $\Delta_{\frac{1}{2}}$ and the
segment $l_{e^+\left(\frac{p}{2^m}\right)e\left(\frac{p+1}{2^m}\right)}$.

\begin{defn}[Le Potier Curve]
In the $\cccp$, consider the open region below all the line segments $l_{e^+\left(\frac{p}{2^m}\right)e^l\left(\frac{p}{2^m}\right)}$,
$l_{e^r\left(\frac{p}{2^m}\right)e^+\left(\frac{p}{2^m}\right)}$ and the curve $\Delta_{\frac{1}{2}}$. The boundary of this open region is a fractal curve in the region between $\Delta_{\frac{1}{2}}$ and $\Delta_{1}$ consisting of line segments $l_{e^+\left(\frac{p}{2^m}\right)e^l\left(\frac{p}{2^m}\right)}$,
$l_{e^r\left(\frac{p}{2^m}\right)e^+\left(\frac{p}{2^m}\right)}$ for all dyadic numbers
$\frac{p}{2^m}$ and fractal pieces of points on $\Delta_{\frac{1}{2}}$. We call this curve the \emph{Le
Potier curve} on the $\cccp$,
and denote it by $C_{LP}$. We call the cone in $\KK_{\mathbb
R}(\pp)$ spanned by the origin and $C_{LP}$ as the \emph{Le Potier
cone}. 
\label{lpcurve}
\end{defn}
We also make a notation for the following open region above
$C_{LP}$.
\[\Geo_{LP}:= \{(1,a,b)\; |\; (1,a,b) \text{ is above }C_{LP}
\text{ and  not on any segment } l_{ee^+}\}.\] 

\begin{theorem}[Drezet, Le Potier]
There exists a Gieseker semistable coherent sheaf with character
$(\ch_0(>0),\ch_1,\ch_2)\in \mathrm K(\pp)$ if and only if either:
\begin{enumerate}
\item it is proportional to an exceptional character $e\left(\frac{p}{2^m}\right)$;
\item The point $\left(1,\frac{\ch_1}{\ch_0},\frac{\ch_2}{\ch_0}\right)$ is on or below $C_{LP}$ in the $\cccp$.
\end{enumerate}
\label{thm:lp}
\end{theorem}

\begin{center}

\begin{tikzpicture}[domain=1:5]

\tikzset{%
    add/.style args={#1 and #2}{
        to path={%
 ($(\tikztostart)!-#1!(\tikztotarget)$)--($(\tikztotarget)!-#2!(\tikztostart)$)%
  \tikztonodes},add/.default={.2 and .2}}
}

\newcommand\XA{0.02}

\draw [name path =C0, opacity=0.1](-3,4.5) parabola bend (0,0) (3,4.5)
 node[right, opacity =0.5] {$\db_0$};

\draw [name path = C1, opacity=0.5](-3,4) parabola bend (0,-0.5) (3,4)
 node[right] {$\db_{\frac{1}{2}}$};

\draw [name path =C2, opacity=0.5](-3,3.5) parabola bend (0,-1) (3,3.5)
 node[right] {$\db_1$};


\coordinate (B3) at (-3,4.5);
\coordinate (B2) at (-2,2);

\coordinate (B1) at (-1,0.5);
\coordinate (A0) at (0,0);
\coordinate (A1) at (1,0.5);
\coordinate (A2) at (2,2);
\coordinate (A3) at (3,4.5);

\coordinate (E0) at (0,-1);
\draw (E0) node [below right] {$e^+(0)$};

\coordinate (E1) at (1,-0.5);
\draw (E1) node [below right] {$e^+(1)$};

\coordinate (E2) at (2,1);
\draw (E2) node [below right] {$e^+(2)$};

\coordinate (E3) at (3,3.5);
\draw (E3) node [below right] {$e^+(3)$};

\coordinate (F1) at (-1,-0.5);
\draw (F1) node [below left] {$e^+(-1)$};

\coordinate (F2) at (-2,1);
\draw (F2) node [below left] {$e^+(-2)$};

\coordinate (F3) at (-3,3.5);
\draw (F3) node [below left] {$e^+(-3)$};

\draw [name path =L0,opacity =\XA] (E0) -- (B2);
\draw [name intersections={of=C1 and L0},  thick] (E0) -- (intersection-1);
\draw [name path =R13,opacity =\XA] (F3) -- (B2);
\draw [name intersections={of=C1 and R13},  thick] (F3) -- (intersection-1);

\draw [name path =L1,opacity =\XA] (E1) -- (B1);
\draw [name intersections={of=C1 and L1},  thick] (E1) -- (intersection-1);
\draw [name path =R12,opacity =\XA] (F2) -- (B1);
\draw [name intersections={of=C1 and R12},  thick] (F2) -- (intersection-1);

\draw [name path =L2,opacity =\XA] (E2) -- (A0);
\draw [name intersections={of=C1 and L2},  thick] (E2) -- (intersection-1);
\draw [name path =R11,opacity =\XA] (F1) -- (A0);
\draw [name intersections={of=C1 and R11},  thick] (F1) -- (intersection-1);

\draw [name path =L3,opacity =\XA] (E3) -- (A1);
\draw [name intersections={of=C1 and L3},  thick] (E3) -- (intersection-1);
\draw [name path =R0,opacity =\XA] (E0) -- (A1);
\draw [name intersections={of=C1 and R0},  thick] (E0) -- (intersection-1);

\draw [name path =L12,opacity =\XA] (F2) -- (B3);
\draw [name intersections={of=C1 and L12},  thick] (F2) -- (intersection-1);
\draw [name path =L11,opacity =\XA] (F1) -- (B2);
\draw [name intersections={of=C1 and L11},  thick] (F1) -- (intersection-1);

\draw [name path =R1,opacity =\XA] (E1) -- (A2);
\draw [name intersections={of=C1 and R1},  thick] (E1) -- (intersection-1);
\draw [name path =R2,opacity =\XA] (E2) -- (A3);
\draw [name intersections={of=C1 and R2},  thick] (E2) -- (intersection-1);

\coordinate (S3) at (-2.5,2.5);
\draw (S3) node [below left] {$e^+(-\frac{5}{2})$};

\coordinate (S2) at (-1.5,0.5);

\coordinate (S1) at (-.5,-0.5);

\coordinate (T1) at (.5,-0.5);

\coordinate (T2) at (1.5,0.5);

\coordinate (T3) at (2.5,2.5);
\draw (T3) node [below right] {$e^+(\frac{5}{2})$};

\draw [name path =RS3,opacity =\XA] (S3) -- (B3);
\draw [name intersections={of=C1 and RS3},  thick] (S3) -- (intersection-1);
\draw [name path =LS3,opacity =\XA] (S3) -- (A0);
\draw [name intersections={of=C1 and LS3},  thick] (S3) -- (intersection-1);

\draw [name path =RS2,opacity =\XA] (S2) -- (B2);
\draw [name intersections={of=C1 and RS2},  thick] (S2) -- (intersection-1);
\draw [name path =LS2,opacity =\XA] (S2) -- (A1);
\draw [name intersections={of=C1 and LS2},  thick] (S2) -- (intersection-1);

\draw [name path =RS1,opacity =\XA] (S1) -- (B3);
\draw [name intersections={of=C1 and RS1},  thick] (S1) -- (intersection-1);
\draw [name path =LS1,opacity =\XA] (S1) -- (A2);
\draw [name intersections={of=C1 and LS1},  thick] (S1) -- (intersection-1);

\draw [name path =RT1,opacity =\XA] (T1) -- (B2);
\draw [name intersections={of=C1 and RT1},  thick] (T1) -- (intersection-1);
\draw [name path =LT1,opacity =\XA] (T1) -- (A3);
\draw [name intersections={of=C1 and LT1},  thick] (T1) -- (intersection-1);

\draw [name path =RT2,opacity =\XA] (T2) -- (B1);
\draw [name intersections={of=C1 and RT2},  thick] (T2) -- (intersection-1);
\draw [name path =LT2,opacity =\XA] (T2) -- (A2);
\draw [name intersections={of=C1 and LT2},  thick] (T2) -- (intersection-1);

\draw [name path =RT3,opacity =\XA] (T3) -- (A0);
\draw [name intersections={of=C1 and RT3},  thick] (T3) -- (intersection-1);
\draw [name path =LT3,opacity =\XA] (T3) -- (A3);
\draw [name intersections={of=C1 and LT3},  thick] (T3) -- (intersection-1);

\draw (0,0) node{$\bullet$}--(0,-1)node{$\bullet$}--(1,0.5)node{$\bullet$}--(2,1)node{$\bullet$}--(2,2)node{$\bullet$}--(0,0);

\draw [->] (0.5,3.5) node[above]{$\cE=\langle \cO,\cO(1),\cO(2) \rangle$}->(1.5,1.1);
\draw (0.5,4.2) node{$\MZ_{\cE}$ for };


\draw[->,opacity =0.3] (-4,0) -- (4,0) node[above right] {$\frac{\ch_1}{\ch_0}$};

\draw[->,opacity=0.3] (0,-2)-- (0,0) node [above right] {O} --  (0,5) node[right] {$\frac{\ch_2}{\ch_0}$};

\end{tikzpicture}\\
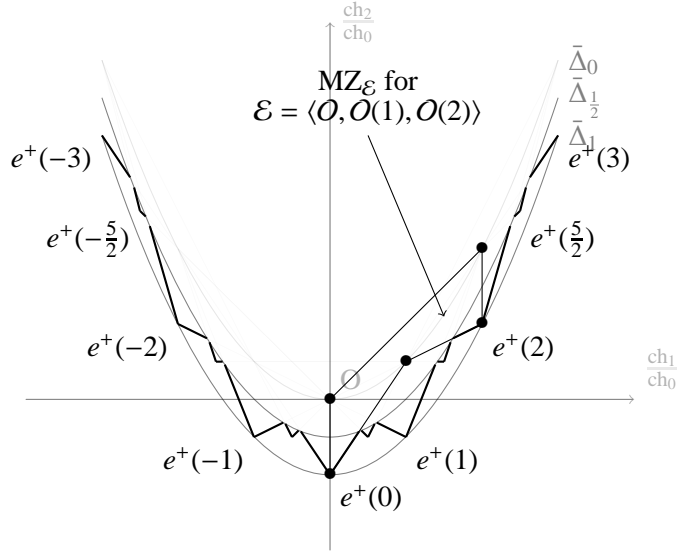
Figure: The Le Potier curve  $C_{LP}$.
\end{center}

\begin{rem}
In this article, when we talk about the
$\cccp$, we always assume
the $\frac{\ch_1}{\ch_0}$-axis is horizontal and the
$\frac{\ch_2}{\ch_0}$-axis is vertical. The phrase `above' is
translated as `$\frac{\ch_2}{\ch_0}$ coordinates is greater than'.
Other words such as: below, right, left can be translated in a similar
way.
\end{rem}

The full strong exceptional collections on $\Db(\pp)$ have
been classified by Gorodentsev and Rudakov \cite{GorRu}. In
particular, up to a cohomological shift, the collection consists of
exceptional bundles on $\textbf P^2$. In terms of dyadic numbers,
their labels are of three cases:
\begin{equation*}
\left\{\frac{p-1}{2^m},\frac{p}{2^m},\frac{p+1}{2^m}\right\};\;
\left\{\frac{p}{2^m},\frac{p+1}{2^m},\frac{p-1}{2^m}+3\right\};\;
\left\{\frac{p+1}{2^m}-3,\frac{p-1}{2^m},\frac{p}{2^m}\right\}.
\tag{$\clubsuit$} \label{eq:dyadictriples}
\end{equation*}

\subsection{Review:Geometric stability conditions}

We briefly recall the definition of stability condition on a triangulated category from \cite{Bri07}. Let $\cT$ be the bounded derived category of coherent sheaves on a smooth variety. A \emph{pre-stability condition} $\sigma=(\cP,Z)$ on $\cT$ consists of a central charge $Z:$ K$_0(\cT) \rightarrow \C$, which is an $\R$-linear homomorphism, and a slicing $\cP: \R\rightarrow$ (full additive subcategories of $\cT$), satisfying the following axioms:
\begin{enumerate}
\item For any object $E$ in $\cP(\phi)$,  we have $Z(E)=m(E)\exp(i\pi\phi)$ for some $m(E)\in\R_{>0}$;
\item $\cP(\phi+1)=\cP(\phi)[1]$;
\item when $\phi_1>\phi_2$ and $A_i\in$ obj($\cP(\phi_i)$), we have $\Hom_{\cT}(A_1,A_2)=0$;
\item (Harder-Narasimhan filtration) For any object $E$ in $\cT$, there is a sequence of real numbers $\phi_1>\cdots >\phi_n$ and a collection of vanishing triangles $E_{j-1}\rightarrow E_j\rightarrow A_j$ with $E_0=0$, $E_n=E$ and $A_j\in\mathrm{obj} \cP(\phi_j)$ for all $j$.
\end{enumerate}

A pre-stability condition is called a \emph{stability condition} if it satisfies the \emph{support condition}:
there exists a quadratic form $Q$ on the vector space K$_{\R}(\cT)$ such that
\begin{itemize}
\item For any $E\in\mathrm{obj}\cP(\phi)$, $Q(E)\geq 0$;
\item $Q|_{\Ker Z}$ is negative definite.
\end{itemize}

For the rest part of this section, we will follow the line of \cite{Bri08} and \cite{BM}
and conclude that the space of geometric stability condition on \textbf P$^2$ is
a $\glt$ fiber space over Geo$_{LP}$.
\begin{defn}
A stability condition $\sigma$ on $\Db(\pp)$ is called
\emph{geometric} if all
skyscraper sheaves $k(x)$ are $\sigma$-stable with the same phase.
We denote the subset of all geometric stability condition by
$\Stab^{\Geo}(\pp)$. \label{defn:stabgeo}
\end{defn}

Let $s$ be a real number, a torsion pair of coherent sheaves on
$\textbf P^2$ is given by:
\begin{itemize}
\item[]Coh$_{\leq s}$: the subcategory of Coh(\textbf P$^2$) generated by
slope semistable sheaves of slope $\leq s$ by extension.
\item[] Coh$_{>
s}$: the subcategory of Coh(\textbf P$^2$) generated by slope semistable
sheaves of slope $> s$ and torsion sheaves.

\item[] Coh$_{\# s}$ $:=$ $\langle\Coh_{\leq s}[1]$, Coh$_{>
s}\rangle$
\end{itemize}
\begin{defn}
Given $(s,q)$ $\in$ $\Geo_{LP}$, the 
$\sigma_{s,q}=(Z_{s,q},\cP_{s,q})$ on $\Db(\pp)$ is defined by the \emph{central charge}
$Z_{s,q}$ on the heart $\cP_{s,q}\left((0,1]\right)=\Coh_{\# s}$.
\[ Z_{s,q}(E):=(-\ch_2(E)+q\cdot \ch_0(E))+ i(\ch_1(E)-s\cdot \ch_0(E)).\]
Let the \emph{phase function} $\phi_{s,q}$ be defined for objects in
$\Coh_{\#s}$:
 $\phi_{s,q}(E) := (1/\pi) \arg(Z_{s,q}(E))$. 
 For $\phi\in(0,1]$, each slice $\cP(\phi)$ is formed by the semistable objects (with respect to $Z_{s,q}$) with phase $\phi_{s,q}=\phi$.
\label{defn:geostabsq}
\end{defn}
\begin{rem}
This definition of  the central charge $Z_{s,q}$  is slightly different from the
usual case as that in the \cite{ABCH}. The imaginary parts are
defined in the same way, but the real part is different from the
usual case by a scalar times the imaginary part. We would like to
use the version here because its kernel is
clear. In addition, if we write $P$ for the point $(1,s,q)$, then the phase (times
$\pi$) of an object $E$ in $\Coh_{\#s}$ is the angle spanned by the
rays $l^+_{PE}$ and $l_{P-}$ (for definition, see Table
\ref{table:listnot}) at $P$ on the $\mathcal H_P$ half plane.
\label{rem:anglephase}
\end{rem}

\begin{prop}
For any $(s,q)$ $\in$ $\Geo_{LP}$, $\sigma_{s,q}=(Z_{s,q},\cP_{s,q})$ is a geometric
stability condition. \label{prop:sigmasqisgeostab}
\end{prop}
For the proof, readers are referred to the arguments in \cite{Bri08} and \cite{BM} Corollary 4.6, which also work
well in the $\pp$ case. Up to the $\glt$-action, geometric
stability conditions can only be of the form given in Proposition
\ref{prop:sigmasqisgeostab}.

\begin{nota}
Given a point $P=(1,s,q)$ in $\Geo_{LP}$, we will also write $\sigma_P$, $\phi_P$, $\Coh_{P}(\pp)$ and $Z_P$ for the stability condition $\sigma_{s,q}$, the phase function $\phi_{s,q}$, the tilt heart $\Coh_{\#s}(\pp)$ and the central charge $Z_{s,q}$ respectively. 
\label{rem:sigmap}
\end{nota}

\begin{prop}[\cite{Bri08} Proposition 10.3, \cite{BM} Section 3] Let $\sigma=(Z,\cP)$ be a
geometric stability condition with all skyscraper sheaves $k(x)$ in
$\mathcal P(1)$. Then the heart $\mathcal P((0,1])$ is $\Coh_{\#s}$
for some real number $s$.  The central charge $Z$ can be written in
the form of
\[-\ch_2+a\cdot \ch_1+b\cdot \ch_0.\] The complex numbers $a$ and $b$
satisfies the following conditions:
\begin{itemize}
\item $\Im a$ $>$ $0$, $\frac{\Im  b}{\Im a}$ $=$ $s$;
\item $(\frac{\Im b}{\Im a}$, $\frac{\Re a\Im b}{\Im a}$+$\Re b)$ is in $\Geo_{LP}$.
\end{itemize}
\label{thm:geostab}
\end{prop}
Knowing the classification result of stable characters, Theorem \ref{thm:lp}, the property is proved in the same way as that in the
local \textbf P$^2$ and K3 surfaces case.

\subsection{Destabilizing walls}
We collect some small but useful lemmas in this section.
\begin{defn}
We call a stability condition \emph{non-degenerate} if the image of
its central charge is not contained in a real line. We write $\Stab^{\mathrm{nd}}(\pp)$ for all the
non-degenerate stability conditions.
\end{defn}
Note that by Proposition \ref{thm:geostab}, $\Stab^{\Geo}(\pp)\subset \Stab^{\mathrm{nd}}(\pp)$.
In this Picard rank $1$ case, the kernel map on the central charge
is well-defined on Stab$^{\mathrm{nd}}(\pp)$. \[\Ker: \Stab^{\mathrm{nd}}(\pp)\rightarrow
\pkp.\]
\begin{lemma}
$\glt$ acts freely on $\Stab^{\mathrm{nd}}(\pp)$ with
closed orbits, and 
\[\Ker:\Stab^{\mathrm{nd}}(\pp) / \glt \rightarrow \pkp\]
 is a
local homeomorphism. \label{lemma:localhomeo}
\end{lemma}
\begin{proof}
By Theorem 1.2 in \cite{Bri07}, Stab$^{\mathrm{nd}}$ $\rightarrow$ Hom$_{\mathbb
Z}(\mathrm{K}(\pp),\mathbb C)$ is a local homeomorphism. The image is
in the non-degenerate part of Hom$_{\mathbb Z}(\KK(\textbf
P^2),\mathbb C)$. Hom$^{nd}_{\mathbb Z}(\KK(\textbf P^2),\mathbb
C)/$GL$^+(2,\mathbb R)$ is just the quotient Grassmannian Gr$_2(3)$
as a topological space.
\end{proof}
\begin{cor}
$\glt$ acts freely on $\Stab^{\Geo}(\pp)$, and 
\[\Stab^{\Geo}(\pp) / \glt \simeq \Geo_{\mathrm{LP}}.\]
\label{cor:geost}
\end{cor}
\begin{lemma}
Let $Z$ be  a non-degenerate central charge, $v$ and $w$  be two non-zero
characters,  then
\[Z(v)\varparallel Z(w)\]
 if and only if $v$, $w$ and the line $\Ker Z$ in $\KK_{\R}(\pp)$ spans a two-dimensional plane.
\label{lemma:paraandspanplane}\end{lemma}
\begin{proof}
$Z(v)\varparallel Z(w)$ if and only if $Z(av+bw)$  $=$ $0$ for some
$a,b\in\mathbb R$. $v$, $w$, $av+bw$ and $O$ are on the same plane.
\end{proof}

\begin{lemma}
Let $P$ be a point in $\Geo_{LP}$, $E$ and
$F$ be two objects in $\Coh_{P}$. The phase
\[\phi_{P}(E)>\phi_{P}(F)\] if and only if the ray $l^+_{PE}$ is
above $l^+_{PF}$. \label{lemma:slopecompare}
\end{lemma}
\begin{proof}
By the definition of $l^+_{PE}$, $l_{P-}$, $Z_{P}$ and Remark \ref{rem:anglephase}, the angle spanned by the rays $l^+_{PE}$ and $l_{P-}$ at point $P$ on the $\cccp$ is
$\pi\phi_{P}(E)$. The statement is clear.
\end{proof}

\begin{prop}
Let $E$ be an $\sigma_P$-stable object, then one of the
following cases will hold:
\begin{enumerate}
\item $\tilde{v}(E)$ is not in the open cone spanned by $\Geo_{LP}$ and the origin.
\item There exists a slope semistable sheaf $F$ such that the
point $P$ is in the region bounded by $l^r_{EF}$ and
$l_{F-}$.
\end{enumerate}
In either case, the line \textbf{\emph{$l_{EP}$ is not inside}}
$\Geo_{LP}$. In particular, at least one of $v(E)$ and $P$ is
outside $\Delta_{<0}$.
\label{prop:epintersectsgeolp}
\end{prop}

\begin{proof}
Suppose $\tilde{v}(E)$ is in the $\Geo_{LP}$-cone, in
particular, $\ch_0$ is not $0$.

When $\ch_0(E)>0$, $\HH^0(E)$ is non-zero. Let $F=\HH^0(E)_{\min}$ be the
quotient sheaf of $\HH^0(E)$ with the minimum slope. Let $D$ be
H$^{-1}(E)$ and $G$ be the kernel of $\HH^0(E)\rightarrow F$. We have
$\mu(D)<\mu(F)<\mu(G)$,
when $D$ and $G$ are non-zero. We have the relation
\[\frac{\ch_1(E)}{\ch_0(E)}\; = \; \frac{\ch_1(F)+\ch_1(G)-\ch_1(D)}{\ch_0(F)+\ch_0(G)-\ch_0(D)}\; \geq\; \frac{\ch_1(F)}{\ch_0(F)}.\]

The equality only holds when $D$ and $G$ are both zero, but this is
not possible as else $v(E)$ $=$ $v(F)$ and  is inside $\Geo_{LP}$ by Theorem \ref{thm:lp}.
Therefore, $v(F)$ is to the left of $v(E)$ on the $\cccp$. Let $P=(1,s,q)$, as $F\in\Coh_{>
s}$, $P$ is to the left of $v(F)$. In addition, as
$\phi_{P}(E)$ $<$ $\phi_{P}(F)$, by Lemma
\ref{lemma:slopecompare}, $P$ is below the line $L_{EF}$. Therefore, $P$ is in the region bounded by $l^r_{EF}$ and $l_{F-}$.

When $\ch_0(E)<0$, let $F=\HH^{-1}(E)_{\max}$ be the subsheaf of
$\HH^{-1}(E)$ with maximum Mumford slope. By the same argument, $v(F)$
is to the right of $v(E)$. As $F\in \Coh_{\leq s}$, $P$ is to
the right of $v(F)$ or on the line $L_{F(0,0,1)}$. In
addition, as $\phi_{P}(F[1])$ $<$ $\phi_{P}(E)$, by Lemma
\ref{lemma:slopecompare}, $P$ is below $L_{EF}$. As $l_{F-}$ does
not intersect $\Geo_{LP}$, $P$ is not on $L_{F(0,0,1)}$. Therefore, $P$ is in the region bounded by $l^r_{EF}$ and $l_{F-}$.

For the last statement, the region $\Delta_{<0}$ is bounded by
a parabola and is convex. For any $v(E)$ and $P$ that are
both in the region, $l_{EP}$ is also in the region which is
contained in $\Geo_{LP}$.
\end{proof}

\begin{cor}
Let $E$ be an exceptional bundle, and $P=(1,s,q)$
be a point in $\Geo_{LP}$, then  $E$ is $\sigma_{P}$-stable if $s <
\mu(E)$ and $l_{EP}$ is contained in $\Geo_{LP}$. On the shifted
side, $E[1]$ is $\sigma_{P}$-stable, if $\mu(E) \leq s$ and
$l_{EP}$ is contained in $\Geo_{LP}$. \label{cor:regionofEstab}
\end{cor}
\begin{proof}
Assume $s < \mu(E)$ and $E$ is not $\sigma_{P}$-stable, then there
is a $\sigma_{P}$-stable object $F$ destabilizing $E$. By the
exact sequence:
\[0\rightarrow \HH^{-1}(F)\rightarrow \HH^{-1}(E)\rightarrow
\HH^{-1}(E/F)\rightarrow \HH^0(F)\rightarrow \HH^0(E)\rightarrow
\HH^0(E/F)\rightarrow 0,\]
 we get H$^{-1}(F)\subset \HH^{-1}(E)=0$, and $v(F)$ is between
$L_{P(0,0,1)}$ and $L_{E(0,0,1)}$. As $\phi_{P}(F)\geq \phi_{P}(E)$
by assumption, by Lemma \ref{lemma:slopecompare}, $v(F)$ is in the
region bounded by $l_{P+}$, $l_{PE}$, and $l_{E+}$. As $l_{EP}$ is in
$\Geo_{LP}$, the whole open region bounded by these three segments is also in $\Geo_{LP}$. The whole line segment $l_{FP}$ is contained in $\Geo_{LP}$(unless $v(F)=v(E)$, which implies $E=F$). By Proposition
\ref{prop:epintersectsgeolp}, $F$ is not $\sigma_{P}$-stable,
which is a contradiction. The $s\geq \mu(E)$ case is proved in a
similar way.
\end{proof}

\begin{rem}
The condition `$l_{EP}$ is contained in $\Geo_{LP}$' is also a
necessary condition. Any ray from $v(E)$ only intersects the Le Potier curve
once, and only intersects finitely many $ee^+$ segments. Assume we
are in the $s<\mu(E)$ case and $l_{EP}$ intersects some $ee^+$
segments, we may choose the one (denoted by $F$) with minimum
$\frac{\ch_1}{\ch_0}$ coordinate. The segment $l_{FP}$ is contained in
$\Geo_{LP}$, and the $\phi_{s,q}(F)>\phi_{s,q}(E)$. By \cite{GorRu},
$\Hom(F,E)$ $\neq$ $0$ when $\mu(F)$ $<$ $\mu(E)$. This leads a
contradiction if $E$ is $\sigma_{s,q}$-stable.
\end{rem}

\section {Algebraic stability conditions}
\subsection{Review: Algebraic stability conditions}
\begin{defn}
We call an ordered set $\mathcal E$ $=$ $\{E_1,E_2,E_3\}$ 
\emph{exceptional triple} on $\Db(\pp)$ if $\mathcal E$ is a
full strong exceptional collection of coherent sheaves on
$\Db(\pp)$. \label{def:exceptionaltriple}
\end{defn}
We will write $e^*_i$ for $e^*(E_i)$ as the associated points on the $\cccp$, where $i=1,2,3$ and $*$ could be $+$, $l$, or $r$. By the definition of $e^*$'s, the relation of dyadic numbers (\ref{eq:dyadictriples}), and Serre duality, the points $e^+_1,e^r_1,e_2,e_3$ are collinear on the line of
$\chi(-,E_1)=0$, and $e^+_3,e^l_3,e_2,e_1$ are collinear on the line of
$\chi(E_3,-)=0$. 

We are now ready to recall the construction of algebraic stability
conditions with respect to exceptional triples.
\begin{prop} [\cite{Mac} Section 3]
Let $\mathcal E$ be an exceptional triple on $\Db(\pp)$, for
any positive real numbers $m_1$, $m_2$, $m_3$ and real numbers
$\phi_1$, $\phi_2$, $\phi_3$ such that:
\[\phi_1<\phi_2<\phi_3,\text{ and }\phi_1+1<\phi_3.\] There is a
unique stability condition $\sigma$ $=$ $(Z,\mathcal P)$ such that
\begin{enumerate}
\item each $E_j$ is stable with phase $\phi_j$;
\item $Z(E_j)=m_j e^{i\pi \phi_j}$.
\end{enumerate}
\label{prop:thetaE}
\end{prop}

\begin{defn}
Given an exceptional triple $\mathcal E$ $=$ $\{E_1,E_2,E_3\}$ on
$\Db(\pp)$, we write $\Theta_{\mathcal E}$ as the space of all
stability conditions in Proposition \ref{prop:thetaE}.
$\Theta_{\mathcal E}$ is parametrized by
\[\{(m_1,m_2,m_3,\phi_1,\phi_2,\phi_3)\in (\mathbb R_{>0})^3\times
\mathbb R^3\text{ $\big |$
}\phi_1<\phi_2<\phi_3,\phi_1+1<\phi_3\}.\] We make the following
notations for some subsets of $\Theta_{\mathcal E}$.
\begin{itemize}
\item $\Theta^{\triangledown}_{\mathcal E}:=$ $\{ \sigma \in \Theta_{\mathcal E}$ $|$ $\phi_2-\phi_1\leq1$, $\phi_3-\phi_2\leq 1$ and $\phi_3-\phi_1\neq 2\}$;
\item $\Theta^{\Geo}_{\mathcal E}:=$ $\Theta_{\mathcal E}\cap $
$\Stab^{\Geo}$;
\item $\Theta^{\mathrm{Pure}}_{\mathcal E}:=$ $\{ \sigma \in \Theta_{\mathcal E}$ $|$ $\phi_2-\phi_1\geq1$ and $\phi_3-\phi_2\geq1\}$;
\item $\Theta^{+}_{\mathcal E,E_1}:=$ $\{ \sigma \in \Theta_{\mathcal E}$ $|$ $\phi_3-\phi_2<1\}$ $\setminus$ $\Theta^{\Geo}_{\mathcal E}$;
\item $\Theta^{-}_{\mathcal E,E_3}:=$ $\{ \sigma \in \Theta_{\mathcal E}$ $|$ $\phi_2-\phi_1<1\}$ $\setminus$ $\Theta^{\Geo}_{\mathcal
E}$.
\end{itemize}
We denote $\Stab^{\Alg}$ as the union of $\Theta_{\mathcal E}$ for all
exceptional triples on $\Db(\pp)$, and call it the algebraic
stability conditions. \label{def: theta E}
\end{defn}

\begin{lemma}
Let $\mathcal E$ $=$ $\{E_1,E_2,E_3\}$ be an exceptional triple, and
$\sigma$ be a stability condition in $\Theta^{Pure}_\mathcal E$. The
only $\sigma$-stable objects are $E_i[n]$ for $i=1,2,3$ and $n\in
\mathbb Z$. \label{lemma:purepure}
\end{lemma}
\begin{proof}
Let $F$ be a $\sigma$-stable object, we may assume it is in the
heart $\langle E_1[a],E_2[b],E_3\rangle$, where $a,b$ $\in$ $\mathbb
Z$ such that $b\geq 1$ and $a-b\geq 1$. When $b\geq 2$, as
Ext$^1(E_3,-)$ and Ext$^1(-,E_3)$ are $0$ for any other generators,
$F$ is either $E_3$ or belongs to $\langle E_1[a],E_2[b]\rangle$. If
$a-b\geq 2$, then $F$ is either $E_1[a]$, or $E_2[b]$. Else we have
$a=b+1$, and we have the sequence $E_2^{\oplus n_2}[b]\rightarrow
F\rightarrow E_1^{\oplus n_2}[b+1]$ in the heart. As $\phi(E_2[b])\geq\phi(E_1[b+1])$, $F$ is not $\sigma$-stable unless $F$ is
$E_2[b]$ or $E_1[a]$.

We may therefore assume $a=2$ and $b=1$, then $F$ is in the form of
$E_1^{\oplus n_1}\rightarrow E_2^{\oplus n_2}\rightarrow E_3^{\oplus
n_3}$. As $\phi_3\geq\phi_2+1\geq \phi_1+2$, $F$ is either $E_3$, $E_2[1]$ or
$E_1[2]$.
\end{proof}

\subsection{Common areas of geometric and algebraic stability conditions}

Let  $\mathcal E$ $=$ $\{E_1,E_2,E_3\}$ be an exceptional triple,  in this section, we will explain how does the algebraic part $\Theta_\mathcal E$ `glue' on to Stab$^{\Geo}$. We denote TR$_{\mathcal E}$ as the triangle region on $\cccp$ bounded anti-clockwise by line segments $l_{e_1e_2}$,
$l_{e_2e_3}$ and $l_{e_3e_1}$ (the edges $l_{e_1e_2}$,
$l_{e_2e_3}$ are defined to be in the TR$_{\mathcal E}$, the three vertices are not).We denote MZ$_{\mathcal E}$ as the open region on $\cccp$ bounded anti
clockwise by line segments $l_{e_1e^+_1}$,
$l_{e^+_1e_2}$, $l_{e_2e^+_3}$,
 $l_{e^+_3e_3}$ and $l_{e_3e_1}$.

\begin{prop}
Let $\mathcal E$ be an exceptional triple, then we have:
\begin{enumerate}
\item $\Theta^{\triangledown}_{\mathcal E} = \glt\cdot \{\sigma_{s,q}\in\Stab^{\Geo}(\pp)\; |\; (1,s,q)$ $\in$ $\mathrm{TR}_{\cE}\}$. In particular,
$\Theta^{\triangledown}_{\cE}$ is in $\Theta^{\Geo}_{\mathcal
E}$.
\item $\Theta^{\Geo}_{\mathcal E} = \glt\cdot \{\sigma_{s,q}\in\Stab^{\Geo}(\pp) \; | \; (1,s,q)\in
\mathrm{MZ}_{\cE} \}$.
\end{enumerate}
\label{prop:commomareaofalggeo}
\end{prop}

\begin{proof}
We first prove the second statement. As MZ$_{\mathcal E}$ is contained in $\Geo_{LP}$, by Corollary
\ref{cor:regionofEstab}, $E_2$ is $\sigma_{s,q}$-stable for any
point $(1,s,q)$ in MZ$_{\mathcal E}$. As $e^+_1,e^r_1,e_2,e_3$ are
collinear on the line of $\chi(-,E_1)=0$, for any point $P$ in
MZ$_\mathcal E$, $l_{EP}$ is contained in $\Geo_{LP}$. By Corollary
\ref{cor:regionofEstab}, $E_3$ is stable in MZ$_{\mathcal E}$. For
the same reason, $E_1$ is stable MZ$_{\mathcal E}$.

For any $(1,s,q)$ in MZ$_{\mathcal E}$, $E_3$ and $E_1[1]$ are in
the heart of Coh$_{\#s}$. By Lemma \ref{lemma:slopecompare},
$\phi_{s,q}(E_1[1])<\phi_{s,q}(E_3)$, hence
\[\phi_{s,q}(E_3)-\phi_{s,q}(E_1) >1.\]
When $s\geq \mu(E_2)$, $E_3$ and $E_2[1]$ are in the heart
Coh$_{\#s}$, we have \[\phi_{s,q}(E_3)-\phi_{s,q}(E_2)>0.\] As
$(1,s,q)$ is above $l_{e_1e_2}$, by Lemma \ref{lemma:slopecompare},
we also have \[\phi_{s,q}(E_2)-\phi_{s,q}(E_1)>0.\] When
$s<\mu(E_2)$, by a similar argument we also have the same
inequalities for $\phi_{s,q}(E_i)$'s. By Proposition
\ref{prop:thetaE}, we get the embedding
\[\Ker^{-1}(\MZ_\mathcal E)\cap\text{Stab}^{\Geo}
\hookrightarrow \Theta_\mathcal E\cap\text{Stab}^{nd}
\xrightarrow{\Ker} \pkp.
\] 
For $(1,s,q)$ outside the area MZ$_\mathcal E$, we have either
$\phi_{s,q}(E_3)$ $-$ $\phi_{s,q}(E_1)$ $\leq$ $1$, $\phi_{s,q}(E_2)$ $\leq$ $\phi_{s,q}(E_1)$, or $\phi_{s,q}(E_3)$ $\leq $ $\phi_{s,q}(E_2)$. Because either $E_1$ and $E_3$
are in the same heart (when $s>\mu(E_3)$ or $s\leq \mu(E_1)$), or
the slope $\phi_{s,q}(E_1[1])$ is greater than $\phi_{s,q}(E_3)$; or both $E_1[1]$ and $E_2[1]$ are in $\Coh_{\#s}$ but $(1,s,q)$ is below $l_{e_1e_2}$; or both $E_2$ and $E_3$ are in $\Coh_{\#s}$ but $(1,s,q)$ is below $l_{e_2e_3}$.
Hence $\sigma_{s,q}$ is not contained in $\Theta_\mathcal E$, this
finishes the second statement of the proposition.

For the first statement, since $\phi_3-\phi_1$ is not an integer,
$\Theta^{\triangledown}_{\mathcal E}$ $\in$ Stab$^{\mathrm{nd}}$. The image
of Ker$\big(\Theta^{\triangledown}_{\mathcal E}\big)$ is in
Tr$_\mathcal E$. By the previous argument, we also have the embedding
\[\big(\Ker^{-1}(\TR_\cE)\cap\Stab^{\Geo}\big)/\glt
\hookrightarrow \Theta^{\triangledown}_{\mathcal
E}/\glt \xrightarrow{\text{Ker}}
\TR_\mathcal E \subset \pkp.
\]
The map $\Ker$ is local homeomorphism and the composition is an
isomorphism. Since $\Theta^{\triangledown}_{\mathcal E}$ is path
connected, the two maps are both isomorphism. We get the first
statement of the proposition.
\end{proof}

\section {Cell decomposition and contractibility}

\subsection{Neighbor cells of geometric stability conditions}
\begin{propd}[Definition of $\Theta^\pm_E$] Given exceptional triples $\mathcal E$ and $\mathcal E'$ on
$\Db(\pp)$ with the same $E_1$ $=$ $E'_1$ $=$ $E$, then
$\Theta^{+}_{\mathcal E,E_1}$ $=$ $\Theta^{+}_{\mathcal E',E'_1}$.
We denote \emph{this subspace} by $\Theta^+_{E}$. In a similar way, we have
the subspace $\Theta^-_{E}$. \label{prop:defoflegs}
\end{propd}
\begin{proof}
Let the three objects in $\mathcal E$ ($\mathcal E'$) be $E$, $E_2$,
$E_3$ ($E$, $E'_2$, $E'_3$). By \cite{GorRu}, $E'_2$, $E'_3$ is
constructed from $E_2$, $E_3$ by consecutive left or right
mutations. Without loss of generality, we may assume ($E'_2$,
$E'_3$) is just ($L_{E_2}E_3$, $E_2$).

By \cite{Mac} Proposition 3.17, at a point $(\overrightarrow{m},
\overrightarrow{\phi})$ in $\Theta_\mathcal E$, when
$\phi_3<\phi_2+1$, $L_{E_2}E_3$ is stable at the point and
its phase satisfies \[\phi_3-1<\phi(L_{E_2}E_3)<\phi_2.\] 
On the other hand, at
a point $(\overrightarrow{m'}, \overrightarrow{\phi'})$ in
$\Theta_{\mathcal E'}$, when $\phi'_3<\phi'_2+1$, $R_{E'_3}E'_2$
($=$ $E_3$) is stable at this point and \[\phi'_3<\phi(E_3)<\phi'_2+1.\] 
Therefore, the left and right mutation identify the
following two subsets in $\Theta_\mathcal E$ and $\Theta_{\mathcal
E'}$.
\begin{equation}\Theta_{\mathcal E}(\phi_2-\phi_1>1,\phi_3-\phi_2<1)\rightleftharpoons\Theta_{\mathcal E'}(\phi'_3-\phi'_2<1)\label{eq2}\end{equation}
Now by the first statement of Proposition \ref{prop:commomareaofalggeo}, \[\Theta_{\mathcal
E}(\phi_2-\phi_1>1,\phi_3-\phi_2<1)\setminus  \Stab^{\Geo} =
\Theta_{\mathcal E}(\phi_3-\phi_2<1)\setminus (\Stab^{\Geo}\cup
\Theta^{\triangledown}_{\mathcal E}) = \Theta^{+}_{\mathcal
E,E_1}.\] By taking off $\Stab^{\Geo}$ on both sides of (\ref{eq2}), we get $\Theta^{+}_{\mathcal E,E_1} =
\Theta^{+}_{\mathcal E',E_1}$. The $\Theta_E^-$ case is proved in
the same way.
\end{proof}

\begin{rem}
$\Theta^-_{E}$ is a chamber that the skyscraper sheaf $k(x)$ is
destabilized by $E$. $\Theta^+_{E}$ is a chamber that the skyscraper
sheaf $k(x)$ is co-destabilized by $E[1]$. We may also use the
notation $\Theta^{+}_{\mathcal E,E_1}$ in some situations, since it has the chart induced
from $\Theta_{\mathcal E}$. \label{rem:wallofgeo}
\end{rem}

\subsection{Boundary of geometric stability conditions}
\begin{lemma}
Let $E$ and $F$ be two exceptional bundles such that $\mu(E)$ $<$
$\mu(F)$, then $E$ is not stable under any stability condition in
$\Theta_{F}^+$ and $F$ is not stable under any stability condition
in $\Theta_{E}^-$. \label{lemma:stabareaofexcobj}
\end{lemma}
\begin{proof}
Let $\mathcal E$ $=$ $\{E_1,E_2,E_3=E\}$  be an exceptional triple
extended from $E$. We may choose $E_2$ such that $\mu(E_2)$ $<$
$\mu(F)-3$. This can be done because of the correspondence between
dyadic triples (\ref{eq:dyadictriples}) and exceptional triples. In
particular, we may choose dyadic triples (\ref{eq:dyadictriples}) of
the second type for some $q$ large enough.

By \cite{GorRu}, since $\mu(E)<\mu(F)$, Hom$(E,F)\neq
0$. Therefore, we have $\phi(F)>\phi(E)$. On the other hand,  Hom($F,E_1[2])$ $=$ Hom($E_1,F(-3))$ $\neq$ $0$, as by
the choice of $\mathcal E$, $\mu(E_1)$ $<$ $\mu(F(-3))$. Therefore,
$\phi(F)<\phi(E_1[2])$.

Now we have $\phi(E_3)<\phi(F)<\phi(E_1[2])$,  this implies
$\phi_3-\phi_1<2$. The point
$(\overrightarrow{m},\overrightarrow{\phi})$ is in the region
$\Theta_\mathcal E(\phi_3-\phi_1<2,\phi_2-\phi_1<1)$. Since $\Theta^{\triangledown}_{\mathcal E}\subset \Stab^{\Geo}$,  $(\overrightarrow{m},\overrightarrow{\phi})$ is in the region 
$\Theta_\mathcal E(\phi_3-\phi_1<2,\phi_3-\phi_2>1)\subset \Stab^{\mathrm{nd}}$.

We may consider the image $W$ of
Ker$\big(\Theta_\mathcal E(\phi_3-\phi_1<2,\phi_3-\phi_2>1)\big)$ on
$\pkp$ and the wall $\mu(E)$ $=$ $\mu(F)$. By
similar arguments in Proposition \ref{prop:commomareaofalggeo} and
the result of Lemma \ref{lemma:paraandspanplane}, $W$ is connected
and is a `triangle' on the projective $\cccp$. On the
$\cccp$, $W$ is the union
of two regions bounded by $\{l^r_{E_1E_2}$, $l_{E_2E}$, $l^r_{E_1E}\}$
and $\{l^r_{EE_1}$, $l^r_{E_2E_1}\}$ respectively. 

\begin{center}
\tikzset{%
    add/.style args={#1 and #2}{
        to path={%
 ($(\tikztostart)!-#1!(\tikztotarget)$)--($(\tikztotarget)!-#2!(\tikztostart)$)%
  \tikztonodes},add/.default={.2 and .2}}
}
\begin{tikzpicture}[domain=2:1]
\newcommand\XA{0.1}
\newcommand\obe{-0.3}

\coordinate (E) at (0,0);
\node  at (E) {$\bullet$};
\node [below] at (E) {$E$};

\coordinate (EA) at (-1,-0.3);
\node  at (EA) {$\bullet$};
\node [below left] at (EA) {$E_2$};

\coordinate (EB) at (-1.3,0.3);
\node  at (EB) {$\bullet$};
\node [below left] at (EB) {$E_1$};
\node at (0,-2.3) {$\bullet$} node[left] at (0,-2.3) {$e^+$};

\coordinate (F) at (2,3);
\node [left] at (F) {$F$};
\node at (F) {$\bullet$};

\draw (-2.2,1) node {$W$};
\draw (2,-1.5) node {$W$};

\draw [add =-1 and 3,dashed] (EA) to (E) node [right]{$L_{E_2E}$};
\draw (EA) -- (E);

\draw [add =-1 and 3] (EB) to (E) node [right]{$l^r_{E_1E}$};
\draw [add =-1 and 2] (E) to (EB) node [left]{$l^r_{EE_1}$};
\draw [add =0 and 0, dashed] (E) to (EB) node [left]{};

\draw [add =-1 and 3] (EA) to (EB) node [right]{$l^r_{E_2E_1}$};
\draw [add =-1 and 5] (EB) to (EA) node [right]{$l^r_{E_1E_2}$};
\draw [add =0 and 0, dashed] (EA) to (EB) node [right]{};

\draw [add =0 and 3,dashed] (E) to (0,-1) node [left]{$l_{E-}$};

\draw [add =0 and 0.33,dashed] (F) to (E) node {$\bullet$} node [below left] {Q};

\end{tikzpicture} 
\end{center}

As $\mu(E_1)<\mu(F(-3))$, $F$
is above $L_{e^+_1e^r_1E_2E}$, on which $\chi(-,E_1)$ $=$ $0$. The
ray $l^r_{FE}$ is in the angle spanned by $l_{Ee^+_1}$ and $l_{E-}$.
Let $Q$ be the intersection of $L_{EF}$ and $L_{e^+_3e^l_3E_2E_1}$,
then it is on the segment $l_{E_2e^+}$. By the position of the
lines, $L_{EF}\cap W$ $=$ $l_{EQ}$ and it is the only wall on which
$\phi(F)=\phi(E)$. By Proposition \ref{prop:commomareaofalggeo} and
Lemma \ref{lemma:slopecompare}, we have $\phi(E)<\phi(F)$,  when $\Ker Z$ is in
the triangle area TR$_{E_2QE}$; and  we have $\phi(E)>\phi(F)$, when $\Ker Z$ is
 in TR$_{EQe^+}$. As $l_{EQ}$ is the only wall, in $\Theta_\mathcal
E(\phi_3-\phi_1<2,\phi_3-\phi_2>1)$, $\phi(E)<\phi(F)$ if and only
if $\Ker Z$ is in TR$_{E_2QE}$ $\subset$ $\Geo_{LP}$. We get the
contradiction that $F$ cannot be stable at any point in
$\Theta^-_E$.

The $\Theta^+_F$ part is proved in the same way.
\end{proof}

\begin{prop}
Let $E$ and $F$ be two exceptional bundles, then $\Theta^+_E$
$\cap$ $\Theta^+_{F}$ is non-empty if and only if $E=F$. The same
statement holds for $\Theta^-_E$ and $\Theta^-_{F}$. In addition
$\Theta^+_E$ $\cap$ $\Theta^-_{F}$ $=$ $\phi$ for any $E$ and $F$.
\label{prop:legsareunique}
\end{prop}

\begin{proof}
For any stability condition  in $\Theta^\pm_E$, $E$ is stable. Assuming $\mu(E)$ $<$
$\mu(F)$, by Lemma \ref{lemma:stabareaofexcobj}, we only need prove
that $\Theta^-_F\cap\Theta^+_E$ is empty.

When  $\mu(E)+3$ $>$ $\mu(F)$, we may choose an exceptional triple
$\mathcal F$ $=$ $\{F_1,F_2,F_3 = F\}$ being an extension of $F$
such that $\mu(F_1)$ $<$ $\mu(E)$. Such triple exists due to the
correspondence of triples of dyadic numbers and exceptional triples.
By Lemma \ref{lemma:stabareaofexcobj}, $F_1$ is not stable in
$\Theta^+_E$, the intersection $\Theta^-_F\cap\Theta^+_E$ is empty.

When $\mu(E)+3$ $=$ $\mu(F)$, in other words, $E$ $=$ $F(-3)$, we
may choose an exceptional triple $\mathcal F$ $=$ $\{F_1,F_2,F_3 =
F\}$ being an extension of $F$. There is another exceptional triple
$\mathcal E$ $=$  $\{E_1=E, E_2 =F_1,E_3 =F_2\}$. Suppose  $\Theta^-_F$ intersects with $\Theta^+_E$ at a point $(\overrightarrow{m},\overrightarrow{\phi})$ in
$\Theta^+_{\mathcal E,E}$. In particular, $F_3$ is stable, we may assume that  it is in a heart $\langle E[a_1],
E_2[a_2],E_3[a_3]\rangle$ for some integers $a_1$, $a_2$, $a_3$ such
that $a_2-a_1\geq 1$ and $a_3-a_2\geq 1$. Since \[\Hom(F,E_1[2])
=\Hom(F,F(-3)[2])=\mathbb C,\] we have $\phi(F)$ $-$ $\phi(E_1)$
$<2$. As $|\phi(F)-\phi(E[a_1])|<1$, we have $a_1 < 3$, hence
$a_3\leq 0$. Now we have \[\phi(F)-\phi(F_2)=\phi(F)-\phi(E_3)\leq
\phi(F)-\phi(E_3[a_3]) < 1.\] Under the chart of $\Theta_{\mathcal
F}$, this stability condition is in $\Theta_{\mathcal
F}(\phi_3-\phi_2<1)$. By Proposition \ref{prop:commomareaofalggeo},
$\Theta_{\mathcal F}(\phi_3-\phi_2<1)\cap \Theta^-_F$ is empty, we
get $\Theta^-_F\cap\Theta^+_{F(-3)}$ is empty.

The last case is when $\mu(E)+3$ $<$ $\mu(F)$. We may choose an
exceptional triple $\mathcal F$ $=$ $\{F_1,F_2,F_3 = F\}$ being an
extension of $F$ such that $\mu(F_1)$ $>$ $\mu(E)+3$. Again, such
triple exists due to the correspondence of triples of dyadic numbers
and exceptional triples. Let $\mathcal E=$ $\{E_1=E,E_2,E_3\}$ be
an extension of $E$ such that $\mu(E_3)$ $<$ $\mu(F_1)-3$. 

Suppose
$F_i$'s are stable at a point
$(\overrightarrow{m},\overrightarrow{\phi})$ in $\Theta^+_{\mathcal
E,E}$.  We may assume that $F_1$ is in the heart $\langle
E_1[a_1],E_2[a_2],E_3[a_3]\rangle$ for $a_i\in\mathbb Z$ such that
$a_1-a_2\geq 1$ and $a_2-a_3\geq 1$. By \cite{GorRu}, since $\mu(E)<\mu(F_1)$, Hom$(E,F_1)\neq
0$, we have $\phi(F)>\phi(E)$. As $F_1$ and $E_3[a_3]$ are in the same
heart, $\phi(F_1)-\phi(E_3[a_3])<1$, we get $a_3\geq 0$. On the
other hand, Hom($F_1,E_1[2])$ $=$ Hom($E_1,F_1(-3))$ $\neq$ $0$, as by
the choice of $\mathcal E$, $\mu(E_1)$ $<$ $\mu(F(-3))$. Therefore,
$\phi(F_1)<\phi(E_1[2])$. As $F_1$ and $E_1[a_1]$ are in the same heart,
$\phi(E_1[a_1])-\phi(F_1)<1$, we get $a_1\leq 2$. As a result, $F_1$ is
in the heart $\langle E_1[2],E_2[1],E_3\rangle$.

By the same argument,
$F_i$'s are all in the same heart $\langle
E_1[2],E_2[1],E_3\rangle$. As $\phi(F_3)-\phi(F_1)$ $<$ $1$, this
stability condition is not in $\Theta_{\mathcal F}$.

As a conclusion, $\Theta^-_F\cap\Theta^+_E$ is empty when $E\neq F$.
\end{proof}
\begin{cor}
The union of geometric and algebraic stability conditions has the
following decompositions:
\[ \Stab^{\Geo}(\pp)\cup \Stab ^{\Alg}(\pp) = \Stab^{\Geo}(\pp) \coprod
\left(\coprod_{E\text{ exc sheaves}} \Theta^{\pm}_{E}\right) \coprod
\left(\coprod_{\mathcal E\text{ exc triple}}
l_{\Theta^{\mathrm{pure}}_{\mathcal E}}\right).\] \label{cor:celldecomp}
\end{cor}
We are now ready to show Stab$^{\Geo}(\textbf P^2)\cup$
Stab$^{\Alg}(\pp)$ form the whole connected component. To do
this, we need to prove that Stab$^{\Geo}(\textbf P^2)\cup$
Stab$^{\Alg}(\pp)$ has no boundary point. The following
important result is from \cite{Mac05}: the boundary of finitely many
$\Theta_{\mathcal E}$ is contained in Stab$^{Alg}$.
\begin{theorem}[Theorem 4.7  in \cite{Mac05}]
Let $\mathcal E$ be an exceptional triple, we have
\[ \partial
\Theta_{\mathcal E} \subset \Stab^{\Alg}.\]
\label{thm:bndryofthetae}
\end{theorem}
To prove the main result, we also need the following description for
details of the boundary of $\Theta^\pm_E$.
\begin{lemma}
Let $E$ be an exceptional bundle, the boundary of $\Theta^+_E$ (as
well as $\Theta^-_E$) is contained in the union of the boundary of $\Stab^{\Geo}$ and the
boundary of $\Theta^{\mathrm{pure}}_{\mathcal E}$ for exceptional triples
$\mathcal E$ that contain $E$:
\[\partial \Theta^+_E\subset \partial \Stab^{\Geo} \bigcup \left(\coprod_{\cE:\text{ exc triple contains $E$}} \partial \Theta^{\mathrm{pure}}_{\cE}\right)\]
 \label{lemma:boundaryofleg}
\end{lemma}
\begin{proof}
Let $\sigma$ $\in$ Stab$^\dag$ be a point on the boundary of
$\Theta^\pm_E$. By Theorem \ref{thm:bndryofthetae}, $\sigma$ belongs
to $\Theta_\mathcal F$ for some exceptional triple $\mathcal F$ $=$
$\{F_1,F_2,F_3\}$. The point $\sigma$ is not in
$\Theta^{\Geo}_\mathcal F$ as else it has an open neighborhood in
Stab$^{\Geo}$. The point $\sigma$ is also not an inner point of 
$\Theta^{\mathrm{pure}}_\mathcal F$, as else it has an open neighborhood such
that the only stable objects are $F_i[n]$ for $i=1,2,3$ and
$n\in\mathbb Z$. By Proposition \ref{prop:legsareunique}, $\sigma$
is not in the inner point of $\Theta^\pm_\mathcal F$. Hence, 
$\sigma$ is either on the boundary of $\Theta^{\text{pure}}_\mathcal F$ or the
boundary between $\Theta_\cF^{\Geo}$ and $\Theta^\pm_\mathcal F$.
\end{proof}
\begin{prop}
The boundary of the geometric stability conditions is contained in
the space of algebraic stability conditions:
 \[ \partial \Stab^{\Geo} \subset \Stab^{\Alg}.\]
 \label{prop:bdryofstabgeo}
\end{prop}
\begin{proof}
Let $\sigma$ $=$ $(Z,\mathcal P)$ be a stability condition on
$\partial \Stab^{\Geo}$, by the principal of chambers (\cite{BM}
Proposition 3.3) we may assume that the skyscraper sheaf $k(x)$ is a
semistable object with phase
$1$ and destabilized by $F_x$ with the same phase.

I. $\sigma$ is non-degenerate. 

 By Lemma \ref{lemma:localhomeo},
$\Ker Z$ is on the boundary of $\Geo_{LP}$.When $\Ker Z$ is at the
infinity line of $\pkp$, its locus is
$(0,0,1)$ as this is the only asymptotic line of the parabola.
However, $\sigma$ cannot be a stability condition, since $Z(k(x))$
$=$ $Z([0,0,1])$ $=$ $0$, contradicting the fact that $k(x)$ is
semistable on the boundary.  When Ker$Z$ is not at the infinity
line, by Proposition \ref{prop:commomareaofalggeo}, $\sigma$ is either
on the boundary of
$\Theta_\mathcal E$ for an exceptional triple $\mathcal E$ or on the $\db_{\frac{1}{2}}$ but not between any $e^r$ and $e^l$. The first case is due to Theorem \ref{thm:bndryofthetae}. 

\newcommand{\chy}{\frac{\ch_1}{\ch_0}}

The second case is more complicated, we will show that $\sigma$ cannot satisfy the support condition. Let $\Ker Z$ be $(1,s,q)$, then $q=\frac{1}{2}s^2-\frac{1}{2}$. Let $L_1$ be the line on the $\cccp$ across the points $(1,s,q)$ and $\left(1,s-3,\frac{1}{2}(s-3)^2-\frac{1}{2}\right)$. Let $L_2$ be the line $\chy=s$. Let $D_r$ be the set of characters defined as
\[D_r:=\{v\in \KK(\pp)\;  | \; \overline v \text{ is strictly below }L_1\text{ and to the right of }L_2, ||\overline{v}-(1,s,q)||<r\}.\]
Let $v_n$ be a character in $D_{\frac{1}{n}}$, as $v_n$ is below $L_1$, it is below the Le Potier curve. By the classification result of \cite{DP}, $\mathfrak M_{MG}(v_n)$ is non-empty. Adopting the notation in \cite{LZ15}, as $v_n$ is below $L_1$, it is in $\mathfrak R_E$ for exceptional $E$ with $\chy(E)<s$. By the criterion for the last wall in \cite{CHW} or \cite{LZ15}, the stability condition $\sigma$ is above the last wall of $v_n$, in another word, there are $\sigma$-stable objects with character $v_n$. On the other hand, as $||\overline{v}-(1,s,q)||<\frac{1}{n}$ and the $\Ker Z$ is $(1,s,q)$, we have 
\[|Z(v_n)|\lesssim \frac{1}{n} ||v_n||.\]
The stability condition $\sigma$ does not satisfy the support condition. We get the contradiction.

II. $\sigma$ is degenerate.

By \cite{Bri07}, Stab$^{\dag}(\pp)$ $\rightarrow$ Hom$_\Z(\KK(\pp),\C)$ is a local homeomorphism, the degenerate locus has
codimension $2$ in Stab$^{\dag}(\pp)$. By \cite{BM} Proposition 3.3, the
destabilizing wall $W^{k(x)}_P$ for the skyscraper sheaf is of
codimension $1$. As the destabilizing walls are locally finite, we
may assume $\sigma$ is on the boundary of $W^{k(x)}_P$ for a
character $P$ in $\KK_{\R}(\pp)$. By Lemma
\ref{lemma:paraandspanplane}, the kernel of the central charge of any stability condition
on $W^{k(x)}_P\cap$Stab$^{nd}$ is on the line $L_{Pk(x)}$, which is a line parallel to the $\frac{\ch_2}{\ch_0}$-axis on the $\cccp$. As the kernel of $W^{k(x)}_P\cap$Stab$^{nd}$  has
codimension one and is on the boundary of $\Geo_{LP}$, it is the
segment of $l_{ee^+}$ for some exceptional bundle $E$.
$W^{k(x)}_P\cap$Stab$^{nd}$ is contained in the closure of $\Theta_\cE^{\Geo}\cup\Theta_{\mathcal E'}^{\Geo}$ for any $\cE=\{E_1,E_2,E\}$ and $\cE'=\{E,E'_2,E'_3\}$. Therefore, $\sigma$ is contained in the closure of $\Theta_\cE^{\Geo}\cup\Theta_{\mathcal E'}^{\Geo}$. By Theorem
\ref{thm:bndryofthetae},  $\sigma\in\Stab^{\Alg}(\pp)$.
\end{proof}

\subsection{Main result}

\begin{theorem}
The connected component $\sdp$ in $\Stab(\pp)$ that contains the geometric stability conditions is the union
of geometric and algebraic stability conditions.
\[\Stab^\dag(\pp) = \Stab^{\Geo}(\pp) \bigcup \Stab^{\Alg}(\pp).\] \label{thm:stabdag}
\end{theorem} 
\begin{proof}
We show that the boundary of $\coprod_{E\text{ exc sheaf}}
\Theta^{\pm}_{E}$ and $\coprod_{\mathcal E\text{ exc triple}}
\Theta^{\mathrm{pure}}_{\mathcal E}$ is contained in
$\Stab^{\Geo}(\textbf P^2) \bigcup \Stab^\dag(\textbf P^2)$. Together
with Corollary \ref{cor:celldecomp} and \ref{prop:bdryofstabgeo}, we
get the conclusion.

Let $\sigma$ be a stability condition on the boundary of
$\coprod_{\mathcal E\text{ exc triple}}
\Theta^{\mathrm{pure}}_{\mathcal E}$. $\sigma$ has at least three
stable objects $A$, $B$, $C$ to generate the Grothendieck group of
D$^b(\textbf P^2)$. There is an open neighborhood of $\sigma$ at
where $A$, $B$, $C$ are always stable. Since the only stable objects
in $\Theta^{\mathrm{pure}}$ are $E_i[n]$ for $i=1,2,3$ and $n\in \mathbb Z$,
$\sigma$ is on the boundary of at most one $\Theta^{\mathrm{pure}}$. By Theorem
\ref{thm:bndryofthetae}, $\sigma$ is in Stab$^{\Alg}$.

Let $\sigma$ be a stability condition on the boundary of
$\coprod_{E\text{ exc sheaf}} \Theta^{\pm}_{E}$, then for any open
neighborhood $U$ of $\sigma$, $U$ intersect the union of boundaries
of $\Theta^\pm_E$. Now by Lemma \ref{lemma:boundaryofleg}, $U$
intersect the union of $\partial$Stab$^{\Geo}$(\textbf P$^2$) and
$\partial \Theta^{\mathrm{pure}}_{\mathcal E}$. $\sigma$ is on the boundary
of either Stab$^{\Geo}$(\textbf P$^2$) or $\coprod_{\mathcal E\text{
exc triple}} \Theta^{\mathrm{pure}}_{\mathcal E}$. By Proposition
\ref{prop:bdryofstabgeo} and the previous paragraph on the boundary of $\coprod_{\mathcal E\text{
exc triple}} \Theta^{\mathrm{pure}}_{\mathcal E}$, $\sigma$ is in Stab$^{\Alg}(\pp)$.
\end{proof}

\begin{cor}
$\Stab^\dag(\pp)$ is contractible. \label{cor:main}
\end{cor}
\begin{proof}
Each $\Theta^{\mathrm{pure}}_{\mathcal E}$ has an open neighborhood in
$\Theta_{\mathcal E}$ which does not intersects any other
$\Theta^{\mathrm{pure}}_{\mathcal E'}$. For example, one may choose the
region of $\phi_3-\phi_2$ $>$ $\frac{1}{2}$ and $\phi_2-\phi_1$ $>$
$\frac{1}{2}$ in $\Theta_{\mathcal E}$. As Stab$^\dag$ admits a
metric, we may then choose open neighborhoods of
$\Theta^{\mathrm{pure}}_{\mathcal E}$'s which do not intersect with each other.
By Corollary \ref{cor:celldecomp}, Stab$^\dag$(\textbf P$^2$) is
homotopic to its subspace Stab$^{\Geo}$(\textbf P$^2$)$\coprod
(\coprod_{E\text{ exc sheaves}} \Theta^{\pm}_{E})$. Each $\Theta^{+}_{E}$ has an open
neighborhood in Stab$^{\Geo}$(\textbf P$^2$)$\coprod (\coprod_{E\text{ exc sheaves}}
\Theta^{\pm}_{E})$ which does not intersects any other
$\Theta^{\pm}_{E}$. For example, we may choose any exceptional
triple $\mathcal E$ with $E_1$ $=$ $E$ and take
$\Theta^{\Geo}_{\mathcal E}$ $\cup$ $\Theta^{+}_{E}$ as the open
neighborhood. We may contract Stab$^{\Geo}$(\textbf P$^2$)$\coprod
(\coprod_{E\text{ exc sheaves}} \Theta^{\pm}_{E})$ to Stab$^{\Geo}$
which is a contractible space.
\end{proof}

\bibliographystyle{abbrv}\bibliography{stabref}

Chunyi Li \space\space\space\space\space\space\space\space\space Email address: Chunyi.Li@ed.ac.uk 

School of Mathematics, The University of Edinburgh, James Clerk Maxwell Building, The King's Buildings, Mayfield Road, Edinburgh, Scotland EH9 3JZ, United Kingdom

\end{document}